\documentclass[reqno,b5paper]{amsart}
\usepackage{amssymb}
\usepackage{dsfont}
\usepackage{color}

\newtheorem{theorem}{Theorem}[section]

\newtheorem{proposition}[theorem]{Proposition}
\newtheorem{corollary}[theorem]{Corollary}
\newtheorem{lemma}[theorem]{Lemma}

\newtheorem{remark}[theorem]{Remark}
\newtheorem{definition}[theorem]{Definition}
\newtheorem{fact}[theorem]{Fact}

\newtheorem{question}[theorem]{Question}

\newcommand{\T}{\mathbb{T}}
\newcommand{\Z}{\mathbb{Z}}
\newcommand{\N}{\mathbb{N}}

\catcode`\@=12
\def\T{{\mathbb T}}
\def\eps{{\varepsilon}}

\def\Z{{\mathbb Z}}
\def\N{{\mathbb N}}
\def\R{{\mathbb R}}
\def\Q{{\mathbb Q}}
\def\P{{\mathbb P}}

\def\NB{$\clubsuit$}
\makeindex

\begin{document}

\title[$s$-characterised subgroups]
{Statistically characterized subgroups related to some non-arithmetic sequence of integers}
\subjclass[2010]{Primary: 22B05, Secondary: 11B05, 40A05} \keywords{Circle group, characterized subgroup, Natural density, statistical convergence, s-characterized subgroup, arithmetic sequence}

\author{Pratulananda Das}

\address{Department of Mathematics, Jadavpur University, Kolkata-700032, India}
\email {pratulananda@yahoo.co.in}

\author{Ayan Ghosh}

\address{Department of Mathematics, Jadavpur University, Kolkata-700032, India}
\email {ayanghosh.jumath@gmail.com}

\begin{abstract}
Recently, in Das et al. (Mediterr. J. Math. 21 : 164, 2024), characterized subgroups are investigated for some special kind of non-arithmetic sequences. In this note, we study subsequent problems in case of ``statistically characterized subgroups" introduced in Dikranjan et al. (Fund. Math. 249 : 185-209, 2020). The entire investigation emphasizes that these statistically characterized subgroups are mostly larger in size, having cardinality $\mathfrak{c}$, and exhibit behavior that significantly differs from that of classically characterized subgroups. As a consequence, we solve an open problem raised in Dikranjan et al. (Fund. Math. 249 : 185-209, 2020).
\end{abstract}
\maketitle
%
%

\section{Introduction and background}
Throughout $\R$, $\Q$, $\Z$, $\P$ and $\N$ will stand for the set of all real numbers, the set of all rational numbers,
the set of all integers, the set of primes and the set of all natural numbers respectively. The first three are equipped with their usual abelian group structure and the circle group $\T$ is identified with the quotient group $\R/\Z$ of $\R$ endowed with its usual compact topology.

Following \cite{Ka}, we may identify $\T$ with the interval [0,1] identifying 0 and 1. Any real valued function $f$ defined on $\T$ can be identified with a periodic function defined on the whole real line $\R$ with period 1, i.e., $f(x+1)=f(x)$ for every real $x$. When referring to a set $X\subseteq \T$ we assume that $X\subseteq [0,1]$ and $0\in X$ if and only if $1\in X$. For a real $x$, we denote its fractional part by $\{x\}$ and $\|x\|$ the distance from the integers, i.e., $\min\big\{\{x\},1-\{x\}\big\}$.

One motivation for the exploration of the notion of ``characterized subgroups" can be attributed to the examination of sequences of multiples $(a_n\alpha)$ of a non-torsion element $\alpha$ of the circle group $\T$. And these sequences of multiples have deep roots in Number Theory (Weyl's theorem of uniform distribution modulo 1) and in Ergodic Theory (Sturmian sequences and Hartman sets \cite{Wi}). It also plays crucial roles in the advancement of the structure theory of locally compact abelian groups.

At times, the fascination with similar sets arises from a different perspective. This interest is rooted in Harmonic Analysis and is particularly focused on the scenario where these sequences become small, even null sequences and where there have been extensive usage and subsequent studies of so called ``trigonometric thin sets". One should specifically mention ``A-sets" (short for Arbault sets) in \cite{A1}, which happen to be generated by characterized subgroups as its basis (see  \cite{BuKR,El,Ka} for the details and investigations in those directions).

Before proceeding further let us formally present the definition of a {\em characterized subgroup} of the circle group $\T$.
\begin{definition}
Let $(a_n)$ be a sequence of integers, the subgroup
$$
t_{(a_n)}(\T) := \{x\in \T: a_nx \to 0\mbox{ in } \T\}
$$
of $\T$ is called {\em a characterized} $($by $(a_n))$ {\em subgroup} of $\T$.
\end{definition}

The term {\em characterized} appeared much later, coined in \cite{BDS}. Further it is important to note that for practical purpose it is sufficient to work with sequences of positive integers only which has actually been considered throughout the history. Characterized subgroups of $\T$ have been studied widely by many authors with a significant focus on understanding these subgroups when they arise from arithmetic sequences in one direction. Precisely a sequence of positive integers $(a_n)$ is an
{\em arithmetic sequence} if
$$1 = a_0 < a_1 < a_2 < \dots < a_n < \dots ~~\mbox{and}~a_n|a_{n+1}~ \mbox{for every}~n \in \mathbb{N}.$$
In this case the ratio, defined by $b_n=\frac{a_n}{a_{n-1}}$ for $n>0$ (so that $b_1:=a_1$), is an integer. The sequence $(a_n)$ is called $b$-bounded ($b$-divergent) if the sequence $(b_n)$ is bounded (divergent).


A non-arithmetic sequence $(\zeta_n)$ was defined in \cite{DK} as follows:
\begin{equation}\label{eqnonarith}
1,2,4,6,12, 18, 24,  \ldots, n!, 2\cdot n!, 3 \cdot n!, \ldots , n \cdot n!, (n+1)!, \ldots
\end{equation}
It was proved in \cite{DK} that $t_{(\zeta_n)}(\T) = \Q/\Z$. Motivated by this observation, for an arithmetic sequence $(a_n)$ the following general class of non-arithmetic sequences was introduced in \cite{DG8}. Let $(d_n^{a_n})$ be an increasing sequence of integers formed by the elements of the set,
\begin{equation}\label{nonarithdef}
\{ra_k \ : \ 1\leq r< b_{k+1}\}.
\end{equation}
When there is no confusion regarding the sequence $(a_n)$, we simply denote this sequence by $(d_n)$. Note that for $a_n=n!$ corresponding non-arithmetic sequence $(d_n)$ coincides with the sequence $(\zeta_n)$. A through investigation regarding the subgroup $t_{d_n}(\T)$ can be found in \cite{DG8}. The whole investigation reiterates that these characterized subgroups are infinitely generated unbounded torsion countable subgroup of the circle group $\T$.\\

In many cases the subgroup $t_{(a_n)}(\T)$ is rather small, even if the sequence $(a_n)$ is not too dense. This suggests that asking $a_nx \to 0 $ maybe somewhat too restrictive (as has been pointed out in more details in \cite{DDB}). A very natural instinct should be to consider modes of convergence which are more general than the notion of usual convergence and here the idea of natural density came into picture, as motivated by the above mentioned observation, Dikranjan, Das and Bose \cite{DDB} introduced the notion of  statistically characterized  subgroups of $\T$ by relaxing the condition $a_nx \to 0 $ with the condition $a_nx \to 0 $ statistically.

For $m,n\in\mathbb{N}$ and $m\leq n$, let $[m, n]$ denotes the set $\{m, m+1, m+2,...,n\}$. By
 $|A|$ we denote the cardinality of a set $A$. The lower and the upper natural densities of $A \subset \mathbb{N}$ are defined by
$$
\underline{d}(A)=\displaystyle{\liminf_{n\to\infty}}\frac{|A\cap [0,n-1]|}{n} ~~\mbox{and}~~
\overline{d}(A)=\displaystyle{\limsup_{n\to\infty}}\frac{|A\cap [0,n-1]|}{n}.
$$
If $\underline{d}(A)=\overline{d}(A)$, we say that the natural
density of $A$ exists and it is denoted by $d(A)$.
As usual,
$$
\mathcal{I}_d = \{A \subset \mathbb{N}: d(A) = 0\}
$$
 denotes the ideal of ``natural density zero" sets and $\mathcal{I}_d^*$ is the dual filter i.e. $\mathcal{I}_d^* = \{A \subset \mathbb{N}: d(A) = 1\}$.

Let us now recall the notion of statistical convergence in the sense of \cite{F,Fr,S,St,Z} (see also \cite{B1,B2,MSC} for applications to Number Theory and Analysis).

\begin{definition}\label{Def1}
A sequence of real numbers $(x_n)$ is said to converge to a real number $x_0$ statistically if for any $\eps > 0$, $d(\{n \in \mathbb{N}: |x_n - x_0| \geq \eps\}) = 0$.
\end{definition}

It was proved in \cite{S} that $x_n \to x_0$ statistically  precisely when there exists a subset A of $ \N$ of asymptotic density 0, such that $\displaystyle{\lim_{n \in \N \setminus A}} x_n = x_0$.
Over the years, the notion of statistical convergence has been studied in metric spaces using the metric instead of the modulus and then has been extended to general topological spaces using open neighborhoods \cite{MK}. Over the last three decades a lot of work has been done on the notion of statistical convergence primarily because it extends the notion of usual convergence very naturally preserving many of the basic properties but at the same time including more sequences under its purview.

In order to relax the condition $a_nx \to 0 $ it is thus natural to involve the notion of statistical convergence. More precisely, $a_nx \to 0$ statistically means that\NB\footnote{we are repeating here Def. \ref{Def1}.} for every $\eps > 0$ there exists a subset A of $ \N$ of asymptotic density 0, such that $\|a_nx\| < \eps$ for every $n \not \in A$.
Using this notion we can introduce our main definition:
\begin{definition}
For a sequence of integers $(a_n)$ the subgroup
\begin{equation}\label{def:stat:conv}
t^s_{(a_n)}(\T) := \{x\in \T: a_nx \to 0\  \mbox{ statistically in }\  \T\}
\end{equation}
of $\T$ is called {\em a statistically characterised} (shortly, {\em an s-characterised}) $($by $(a_n))$ {\em subgroup} of $\T$.
\end{definition}

The following  result justifies the investigation of this new notion of s-characterised subgroups as it is established that, though in general, larger in size, these subgroups are still essentially topologically nice.
\begin{theorem}\label{theoremborel}\cite[Theorem A]{DDB}
 $t^s_{(a_n)}(\T)$ is a $F_{\sigma\delta}$ (hence, Borel) subgroup of $\T$ containing $t_{(a_n)}(\T)$.
\end{theorem}
This result seems reasonable enough for further investigation of the notion of $s$-characterized subgroups. However in order to really justify that the theory of $s$-characterized subgroups of $\T$ is worthy of further studies and its investigation may not follow from the existing literature on characterized subgroups, we would present instances of sequences $(d_n)$ with non-coinciding $t^s_{(d_n)}(\T)$ and $ t_{(d_n)}(\T)$.

In this article we have considered subgroups that are generated by using the notion of statistical convergence \cite{F}, an important generalization of the notion of usual convergence using natural density \cite{Bu}. These subgroups have been recently introduced in \cite{DDB} and named ``statistically characterized subgroups". They have been studied for arithmetic sequences in \cite{DDB} and later in \cite{DG} and it has been observed that they generate new subgroups of the circle group which cannot be generated in the classical way (i.e., by using usual convergence). Naturally the question arises regarding the cardinality of the statistically characterized subgroup generated by the sequence $(d_n)$. Here in Theorem \ref{sconth04}, we have shown that for arithmetic sequences satisfying some properties they are always of size $\mathfrak{c}$. This in turn answers an open problem Question 2.18 (posed in \cite{DDB}). As far as the comparison between the statistically characterized subgroups for the sequences $(a_n)$ and $(d_n)$ is concerned, one obtains exactly similar observation like the classical case (Theorem 2.16) for $b$-bounded arithmetic sequences but otherwise the picture seems to be much more complex where no particular conclusion can be drawn.

Throughout the article by $(d_n)$ we would always mean the sequence defined in Eq (\ref{nonarithdef}) corresponding to the arithmetic sequence $(a_n)$ unless otherwise stated.

\vspace{.1cm}
\section{Main results.\vspace{.3cm} \\ }
The primary aim of this section is to understand the behavior of $s$-characterized subgroups for the sequence $(a_n)$ and $(d_n)$ and we will observe in due course that here the situation is much more complicated unlike the classical characterized subgroups.

For arithmetic sequences, the following facts will be used in this sequel time and again. So, before moving onto our main results here we recapitulate that once.
\begin{fact}\label{lemmanew}\cite{DI1}
For any arithmetic sequence $(a_n)$ and $x\in\T$, we can find a unique sequence $c_n\in [0,b_n-1]$ such that
\begin{equation}\label{canonical:repr}
x=\sum\limits_{n=1}^{\infty}\frac{c_n}{a_n},
\end{equation}
where $c_n<b_n-1$ for infinitely many $n$.
\end{fact}
\begin{proof}
For better clarity we recall the construction of the sequence $(c_n)$. Consider $c_1=\lfloor a_1x\rfloor$, where $\lfloor \ \rfloor$ denotes the integer part. Therefore, $x-\frac{c_1}{a_1}<\frac{1}{a_1}$.

Suppose, $c_1,c_2,\ldots,c_k$ are defined for some $k\geq 1$ with $x_k=\sum\limits_{n=1}^{k}\frac{c_n}{a_n}$ and $x-x_k<\frac{1}{a_k}$. Then the $(k+1)$-th element is defined as $c_{k+1}=\lfloor a_{k+1}(x-x_k)\rfloor$.
\end{proof}

For $x\in\T$ with canonical representation (\ref{canonical:repr}), we define
\begin{itemize}
\item[$\bullet$] $supp_{(a_n)}(x) = \{n\in \N: c_n \neq 0\}$,
\item[$\bullet$] $supp^b_{(a_n)}(x)=\{n\in\N\ : \ c_n=b_n-1\}$.
\end{itemize}
For an arithmetic sequence $(a_n)$, note that for each $j\in\N$,
\begin{equation}\label{eqsum}
\sum\limits_{i=j}^\infty \frac{c_i}{a_i} \leq \sum\limits_{i=j}^\infty \frac{b_i-1}{a_i} = \sum\limits_{i=j}^\infty \bigg(\frac{1}{a_{i-1}} - \frac{1}{a_i} \bigg) \leq \frac{1}{a_{j-1}}.
\end{equation}

\begin{remark}\label{uconr1}
Observe that $(a_n)$ is a subsequence of $(d_n)$. Therefore, we can write $a_k=d_{n_k}$. Now from construction of the sequence $(d_n)$, it is evident that $n_{k+1}-n_k=b_{k+1}-1$ and $d_{(n_k+r-1)}=ra_k$.
\end{remark}

We start with an observation which is in line with \cite[Theorem 2.4]{DG8} that is presented in Theorem \ref{sconth1}. Before that we take note of the following two easy lemmas.

\begin{lemma}\label{sconl00}
For any increasing sequence of integers $(u_n)$, if $d(\{n_k:k\in\N\})=1$ then $t^s_{(u_n)}(\T)= t^s_{(u_{n_k})}(\T)$.
\end{lemma}

\begin{lemma}\label{sconl01}
For any increasing sequence of integers $(u_n)$, if $\overline{d}(\{n_k:k\in\N\})>0$ then $t^s_{(u_n)}(\T)\subseteq t^s_{(u_{n_k})}(\T)$.
\end{lemma}

\begin{theorem}\label{sconth1}
If $(a_n)$ is $b$-bounded then $t^s_{(a_n)}(\T)=t^s_{(d_n)}(\T)$.
\end{theorem}

\begin{proof}
If $(a_n)$ is $b$-bounded then we can choose $M\in\N$ such that $2\leq b_n\leq M$ for each $n\in \N$. In view of Remark \ref{uconr1}, we also have
\begin{equation}\label{sconeq1}
a_k=d_{n_k} \ \mbox{   and   }  \  n_{k+1}-n_k=b_{k+1}-1 < M.
\end{equation}
\begin{itemize}
\item[$\bullet$]  Let us write $A=\{n_k:k\in\N\}$. If possible assume that $d(A)=0$. Now Eq (\ref{sconeq1}) implies that $\N\subseteq\bigcup\limits_{i=0}^M A+i$ and consequently we must have $d(\N)\leq \sum\limits_{i=0}^M d(A+i)=0$ $-$ which is a contradiction. So, we conclude that $\overline{d}(A)>0$. Therefore Lemma \ref{sconl01} ensures that $t^s_{(d_n)}(\T)\subseteq t^s_{(a_n)}(\T)$.

\item[$\bullet$]  Next let $x\in t^s_{(a_n)}(\T)$ and $\varepsilon >0$ be given. Then there exists $A'\subseteq \N$ with $d(A')=0$ and for all $n\in\N\setminus A'$ we have $\|a_nx\|<\frac{\varepsilon}{M}$.\\
    Let us write $B'=\{n_k:k\in A'\}$. Then it is evident that $n_k\geq k$. Now observe that
    \begin{eqnarray*}
    \overline{d}(B') &=& \limsup\limits_{n\to\infty} \frac{|\{j\in\N: j\leq n \mbox{  and  } j\in B'\}|}{n} \\ &\leq& \limsup\limits_{n\to\infty} \frac{|\{j\in\N: j\leq n \mbox{  and  } j\in A'\}|}{n} \\ &=& \overline{d}(A')=0.
    \end{eqnarray*}
    We define $B=\bigcup\limits_{i=0}^M B'+i$. Clearly $d(B)=0$. Now observe that whenever $i\notin B$ we have $i=n_{k}+r-1$ for some $k\in \N\setminus A'$.
    Therefore, for all $i\in \N\setminus B$ we have
    $$
    \|d_ix\|=\|d_{n_{k}+r-1}x\| = \|ra_kx\|= r\|a_kx\|<\frac{(M-1)\varepsilon}{M}<\varepsilon.
    $$
    So we can conclude that $x\in t^s_{(d_n)}(\T)$. Evidently we then have $t^s_{(a_n)}(\T)\subseteq t^s_{(d_n)}(\T)$.
\end{itemize}
Thus we conclude that $t^s_{(a_n)}(\T)=t^s_{(d_n)}(\T)$.
\end{proof}

In the rest of this section we investigate the relation between the s-characterized subgroups corresponding to the sequences $(a_n)$ and $(d_n)$ when the sequence is not $b$-bounded. The following results demonstrate that it is not possible to arrive at any concrete conclusion, thus justifying the importance of just proved Theorem \ref{sconth1}.

However as a clear departure from the case of characterized subgroups, the condition ``$b$-boundedness" of the sequence $(a_n)$ is not necessary for the equality here as is demonstrated in the next result.
\begin{proposition}\label{sconp2}
There exists an arithmetic sequence $(a_n)$ such that $(a_n)$ is not $b$-bounded but $t^s_{(a_n)}(\T) = t^s_{(d_n)}(\T)$.
\end{proposition}

\begin{proof}
We construct $A\subseteq\N$ in the following way,
$$
A=\bigcup\limits_{n=1}^\infty [g_n,h_n] \mbox{  and  } \ g_1=1, g_n\leq h_n, |g_{n+1}-h_n|\to\infty \mbox{  and  } d(A)=1
$$
(Note that such a set can always be constructed by taking $|g_{n+1}-h_n|=n$ and $|h_{n}-g_n|=n^2$). Now let us write $A=\{1=n_0<n_1<n_2<\ldots<n_k<\ldots\}$. Then there exists a sequence $(s_n)$ such that $n_{s_j}= h_j$ and $n_{s_j+1}= g_{j+1}$. For each $k\in\N\cup \{0\}$, we define
$$
b_{k+1}=n_{k+1}-n_k +1.
$$
Then for each $i\in\N$, $b_i$ can be written as
$$
b_i=
  \begin{cases}
   g_{j+1}-h_j+1 & \text{if} \ i=s_j+1, \\
   2 & \text{otherwise}.
  \end{cases}
$$
Now the corresponding arithmetic sequence $(a_n)$ is defined by
$$
a_0=1 \mbox{  and  }  a_{n+1}=b_{n+1}a_n.
$$
Note that from the construction of $(b_n)$, it follows that $(a_n)$ is not $b$-bounded and $a_k=d_{n_k}$. Since $d(A)=1$, Lemma \ref{sconl00} ensures that $t^s_{(a_n)}(\T)=t^s_{(d_n)}(\T)$.
\end{proof}

We would like to introduce a condition on an element $x\in\T$ which could ensure that $x\in t^s_{(a_n)}(\T)$ but not in $t^s_{(d_n)}(\T)$. As $t_{(a_n)}(\T)\subsetneq t^s_{(a_n)}(\T)$ we will see in the next result that it is possible to in fact choose $x\in t_{(a_n)}(\T)$ which is outside $t^s_{(d_n)}(\T)$, thus presenting a stronger observation.
\begin{lemma}\label{sconl1}
Let $(a_n)$ be an arithmetic sequence and $x\in \T$ be such that
$$
supp_{(a_n)}(x)=\{s_j+1 : j\in\N \} \ \mbox{  and  } \ c_n=1 \ \mbox{  for all  } \  n\in supp_{(a_n)}(x).
$$
If $supp_{(a_n)}(x)$ is $b$-divergent and $\overline{d} \big(\bigcup\limits_{j=1}^\infty [n_{s_j},n_{s_j}+b_{s_j+1}-2]\big)>0$, then $x\in t_{(a_n)}(\T)\setminus t^s_{(d_n)}(\T)$.
\end{lemma}

\begin{proof}
Let $x\in \T$ be such that
$$
supp_{(a_n)}(x)=\{s_j+1 : j\in\N \} \ \mbox{  and  } \ c_n=1 \ \mbox{  for all  } \  n\in supp_{(a_n)}(x),
$$
where further $supp_{(a_n)}(x)$ is $b$-divergent. Since $\lim\limits_{n\to\infty} \frac{c_n}{b_n} =0$, \cite[Corollary 3.4]{DI1} ensures that $x\in t_{(a_n)}(\T)$.\\
Now proceeding as in , we get
\begin{equation}\label{eqmain2}
\frac{r}{b_{s_j+1}} \leq r\big\{ a_{s_j} x  \big\} \leq \frac{2r}{b_{s_j+1}}.
\end{equation}
Since $supp_{(a_n)}(x)$ is $b$-divergent, the sequence $(b_{s_j+1})$ must diverge to infinity, without loss of any generality we can assume that $b_{s_j+1}\geq 12$ for each $j\in\N$. For each $j\in\N$, let us choose $r\in [\lfloor\frac{b_{s_j+1}}{6}\rfloor,\lfloor\frac{b_{s_j+1}}{3}\rfloor]$, i.e., $$\frac{b_{s_j+1}}{6}-1< r \leq \frac{b_{s_j+1}}{3}.$$
We set
$$
B=\bigcup\limits_{j=1}^\infty [n_{s_j}+\lfloor\frac{b_{s_j+1}}{6}\rfloor,n_{s_j}+\lfloor\frac{b_{s_j+1}}{3}\rfloor].
$$
It is evident that $\overline{d}(B)\geq \frac{1}{6}\overline{d} \big(\bigcup\limits_{j=1}^\infty [n_{s_j},n_{s_j}+b_{s_j+1}-2]\big)>0$. Now observe that for each $i\in B$ we have $i=n_{s_j}+r-1$ for some $r\in [\lfloor\frac{b_{s_j+1}}{6}\rfloor,\lfloor\frac{b_{s_j+1}}{3}\rfloor]$. Therefore for all $i\in B$, from Eq (\ref{eqmain2}) we have
$$
\{d_ix\}=\{d_{n_{s_j}+r-1}x\} = \{ra_{s_j}x\}\leq \frac{2r}{b_{s_j+1}}\leq \frac{2}{3},
$$
and
$$
\{d_ix\}= \{ra_{s_j}x\}\geq \frac{r}{b_{s_j+1}}> \frac{1}{6}-\frac{1}{b_{s_j+1}}\geq \frac{1}{12}.
$$
Since $\overline{d}(B)>0$ and for each $i\in B$ we have $\{d_ix\}\in[\frac{1}{12},\frac{2}{3}]$, we conclude that $x\notin t^s_{(d_n)}(\T)$. Thus, $x\in t_{(a_n)}(\T)\setminus t^s_{(d_n)}(\T)$.
\end{proof}

\begin{proposition}\label{sconp3}
There exists an arithmetic sequence $(a_n)$ such that $(a_n)$ is not $b$-bounded and $t^s_{(d_n)}(\T) \subsetneq t^s_{(a_n)}(\T)$.
\end{proposition}

\begin{proof}
We consider $A\subseteq\N$ such that
$$
A=\bigcup\limits_{n=1}^\infty [g_n,h_n] \mbox{  and  } \ g_1=1, g_n\leq h_n, |g_{n+1}-h_n|\to\infty \mbox{  and  } \overline{d}(A)>0,\overline{d}(\N\setminus A)>0
$$
(Note that such a set can be constructed by taking $|g_{n+1}-h_n|=|h_{n}-g_n|=n$). Now let us write $A=\{1=n_0<n_1<n_2<\ldots<n_k<\ldots\}$. Then there exists a sequence $(s_n)$ such that $n_{s_j}= h_j$ and $n_{s_j+1}= g_{j+1}$. For each $k\in\N\cup \{0\}$, we define
$$
b_{k+1}=n_{k+1}-n_k +1.
$$
Then for each $i\in\N$, $b_i$ can be written as
$$
b_i=
  \begin{cases}
   g_{j+1}-h_j+1 & \text{if} \ i=s_j+1, \\
   2 & \text{otherwise}.
  \end{cases}
$$
Now, the corresponding arithmetic sequence $(a_n)$ is defined by
$$
a_0=1 \mbox{  and  }  a_{n+1}=b_{n+1}a_n.
$$
From the construction of $(b_n)$, it follows that $(a_n)$ is not $b$-bounded and $a_k=d_{n_k}$. Since $\overline{d}(A)>0$, Lemma \ref{sconl01} ensures that $t^s_{(d_n)}(\T)\subseteq t^s_{(a_n)}(\T)$.

Now let us choose $x\in \T$ in the following way,
$$
supp_{(a_n)}(x)=\{s_j+1 : j\in\N \} \ \mbox{  and  } \ c_n=1 \ \mbox{  for all  } \  n\in supp_{(a_n)}(x).
$$
Observe that $supp_{(a_n)}(x)$ is $b$-divergent. Since $\N\setminus A \subseteq \bigcup\limits_{j=1}^\infty [n_{s_j},n_{s_j}+b_{s_j+1}-2]$, we also have
$$
\overline{d} \big(\bigcup\limits_{j=1}^\infty [n_{s_j},n_{s_j}+b_{s_j+1}-2]\big)\geq\overline{d}(\N\setminus A)>0.
$$
Therefore, Lemma \ref{sconl1} ensures that $x\in t_{(a_n)}(\T)\setminus t^s_{(d_n)}(\T)$. Since $ t_{(a_n)}(\T)\subseteq t^s_{(a_n)}(\T)$, we conclude that $x\in t^s_{(a_n)}(\T)\setminus t^s_{(d_n)}(\T)$. Thus, $t^s_{(d_n)}(\T) \subsetneq t^s_{(a_n)}(\T)$.
\end{proof}

Our final results of this section show that in case of arithmetic sequence $(a_n)$ which are not $q$-bounded, both inclusions between the two s-characterized subgroups $t^s_{(a_n)}(\T)$ and $t^s_{(d_n)}(\T)$ may not hold (compare this with \cite[Corollary 2.5]{DG8} for the classical situation).

\begin{theorem}\label{sconth2}
If $(a_n)$ is $b$-divergent then $t_{(a_n)}(\T) \nsubseteq t^s_{(d_n)}(\T)$.
\end{theorem}

\begin{proof}
Let $(a_n)$ be a $b$-divergent arithmetic sequence and $a_k=d_{n_k}$. We choose a sequence $(s_n)$ such that $\overline{d} \big(\bigcup\limits_{j=1}^\infty [n_{s_j},n_{s_j}+b_{s_j+1}-2]\big)>0$ (existence of such a sequence can be assured by considering $s_n=2n$ or $s_n=2n+1$).\\
Let $x\in T$ be such that
$$
supp_{(a_n)}(x)=\{s_j+1 : j\in\N \} \ \mbox{  and  } \ c_n=1 \ \mbox{  for all  } \  n\in supp_{(a_n)}(x).
$$
Since $(a_n)$ is $b$-divergent, it is obvious that $supp_{(a_n)}(x)$ is also $b$-divergent. Therefore, Lemma \ref{sconl1} ensures that $x\in t_{(a_n)}(\T)\setminus t^s_{(d_n)}(\T)$. Thus, $t_{(a_n)}(\T) \nsubseteq t^s_{(d_n)}(\T)$.
\end{proof}

\begin{corollary}\label{sconcoro1}
If $(a_n)$ is $b$-divergent then $t^s_{(a_n)}(\T) \nsubseteq t^s_{(d_n)}(\T)$.
\end{corollary}

\begin{proof}
Since $t_{(a_n)}(\T)\subseteq t^s_{(a_n)}(\T)$, the proof follows directly from Theorem \ref{sconth2}.
\end{proof}

\begin{proposition}\label{sconth3}
There exists an arithmetic sequence $(a_n)$ such that $(a_n)$ is not $b$-bounded and $t^s_{(d_n)}(\T) \nsubseteq t^s_{(a_n)}(\T)$.
\end{proposition}

\begin{proof}
We consider $A=\{s_1<s_2<\ldots<s_n<\ldots\}\subseteq\N$ such that $3$ divides $s_n$ for each $n\in\N$ and $|s_{n+1}-s_n|\geq n$. Since $(s_n)$ is a lacunary sequence (i.e., $|s_{n+1}-s_n|\to\infty$), it is well known that $d(A)=0$. For each $n\in\N$, we define
$$
B=\bigcup\limits_{n=1}^\infty [g_n,h_n] \ \mbox{  where   } \ g_n=s_n+\frac{s_{n+1}-s_n}{3} \  \mbox{  and  }  \  h_n=s_n+\frac{2(s_{n+1}-s_n)}{3}.
$$
Therefore, it is easy to see that $\overline{d}(B)\geq \frac{1}{3}\overline{d}(\N\setminus A)=\frac{1}{3}$.

For each $i\in\N$, we define
$$
b_i=
  \begin{cases}
   (j-2)(s_{j+1}-s_j) & \text{if} \ i=s_j+1, \\
   3 & \text{otherwise}.
  \end{cases}
$$
Now, the corresponding arithmetic sequence $(a_n)$ is defined by
$$
a_0=1 \mbox{  and  }  a_{n+1}=b_{n+1}a_n.
$$
Therefore, from the construction of $(b_n)$, it follows that $(a_n)$ is not $b$-bounded and $a_k=d_{n_k}$. Also observe that
$$
n_{s_j+1}-n_{s_j}=b_{s_j+1} -1=(j-2)(s_{j+1}-s_j)\ \mbox{   and   } \ n_{s_{j+1}}-n_{s_j+1}\leq 2(s_{j+1}-s_j).
$$
Consequently, we have $d\big(\bigcup\limits_{j=1}^\infty [n_{s_j},n_{s_j}+b_{s_j+1}-2]\big)=1$.

Let $x\in T$ be such that
$$
supp_{(a_n)}(x)=B \ \mbox{  and  } \ c_n=1 \ \mbox{  for all  } \  n\in supp_{(a_n)}(x).
$$
Since $b_n=3$ for each $n\in B$, we have $supp_{(a_n)}(x)$ is $b$-bounded. Also observe that $\overline{d}(supp_{(a_n)}(x)\setminus supp^b_{(a_n)}(x))=\overline{d}(B)>0$. Therefore, \cite[Corollary 3.9]{DG2} ensures that $x\notin t^s_{(a_n)}(\T)$.

Now observe that for each $i\in \bigcup\limits_{j=1}^\infty [n_{s_j},n_{s_j}+b_{s_j+1}-2]$ we have $i=n_{s_j}+r-1$ for some $r\in [1,b_{s_j+1}]$. Therefore for all $i\in B$, we have
$$
\{d_ix\}=\{d_{n_{s_j}+r-1}x\} = \{ra_{s_j}x\} = r a_{s_j} \sum\limits_{i\in B, i\geq g_j}^\infty \frac{1}{a_{i}} \leq  \frac{2r a_{s_j}}{a_{g_j}} \leq \frac{2r}{2^{\frac{(s_{j+1}-s_j)}{3}}},
$$

$$
\Rightarrow \lim\limits_{i\in B, i\to\infty} \{d_ix\}\leq \lim\limits_{j\to\infty} \frac{2j(s_{j+1}-s_j)}{2^{\frac{(s_{j+1}-s_j)}{3}}}=0.
$$
Since $d(B)=1$, we conclude that $x\in t^s_{(d_n)}(\T)$. Thus, $t^s_{(d_n)}(\T) \nsubseteq t^s_{(a_n)}(\T)$.
\end{proof}

\section{Cardinality related observation}
In \cite{DDB}, it was shown that for each arithmetic sequence $(a_n)$ the subgroup $t^s_{(a_n)}(\T)$ is always different from $t_{(a_n)}(\T)$ and it has cardinality $\mathfrak{c}$. After that the following open problem was considered (the sequence $(\zeta_n)$ is defined in Eq (\ref{eqnonarith})).

\begin{question} \label{ques2}
Compute $t^s_{(\zeta_n)}(\T)$. Is it countable ? Is it distinct from $\Q/\Z$?
\end{question}

Here we answer Question \ref{ques2} in a more general form considering the sequence $(d_n)$ associated with $(a_n)$. In view of Theorem B \cite{DDB}, Theorem \ref{sconth1} as well as Proposition \ref{sconp2}, we already have instances when $|t^s_{(d_n)} (\T)|=\mathfrak{c}$. However as we will see later, it is not always possible to draw any conclusion about $t^s_{(d_n)} (\T)$ from $t^s_{(a_n)} (\T)$. Indeed it is possible to compute the cardinality of $t^s_{(d_n)} (\T)$ independently which in turn would provide a negative answer to the aforesaid open question.

\begin{theorem}\label{sconth04}
Let $(a_n)$ be an arithmetic sequence such that for each $m\in\N$, $\lim\limits_{n\to\infty}\frac{\sum\limits_{i=0}^{m-1} (b_{n-i}-1)}{\sum\limits_{i=1}^n (b_i-1)}=0$. Then $|t^s_{(d_n)}(\T)|=\mathfrak{c}$.
\end{theorem}

\begin{proof}
Let $(a_n)$ be an arithmetic sequence and $a_k=d_{n_k}$. We construct a sequence $(s_j)$ such that
$$
s_1=1 \ \mbox{   and   } \ s_{j+1}=\min \{r\in\N\ :  \ r> s_j+j+1 \ \mbox{ and } \  \ n_r\geq j  \sum\limits_{i=1}^j \sum\limits_{t=0}^{i-1} (b_{s_i+1-t}-1)\}.
$$

Let $\varepsilon>0$ be given. We choose $m\in\N\setminus \{1,2\}$ such that $\frac{1}{2^{m-2}}<\varepsilon$. Note that for each $j\geq 1$ we have $n_{s_{(j+1)}}\geq j  \sum\limits_{i=1}^j \sum\limits_{t=0}^{i-1} (b_{s_i+1-t}-1)$ and for each $j\geq m$, we have $n_{s_{(j+1)}}> n_{s_j+m+1}$. Set
$$
A=[1,n_{s_m+1-m}-1]\cup \bigcup\limits_{j=m}^\infty [n_{s_j+1-m},n_{s_j+1}-1].
$$
Therefore,
\begin{eqnarray*}
d(A) &\leq &\limsup\limits_{n\to\infty} \frac{|A\cap [1,n]|}{n} \\ &=& \lim\limits_{j\to\infty} \frac{|A\cap [1,n_{s_j+1}-1]|}{n_{s_j+1}-1} \\ &=& \lim\limits_{j\to\infty} \frac{|A\cap [1,n_{s_j+1-m}-1]|}{n_{s_j+1}-1} + \lim\limits_{j\to\infty} \frac{|A\cap [n_{s_j+1-m},n_{s_j+1}-1]|}{n_{s_j+1}-1} \\ &=& \lim\limits_{j\to\infty} \frac{|A\cap [1,n_{s_m+1-m}-1]|}{n_{s_j+1}-1} + \lim\limits_{j\to\infty} \frac{|A\cap [n_{s_m+1-m},n_{s_{j-1}+1}-1]|}{n_{s_j+1}-1} \\ & \ & \ + \lim\limits_{j\to\infty} \frac{\sum\limits_{i=0}^{m-1}(b_{s_j-i+1}-1)}{n_{s_j+1}-1} \\ &\leq& 0 + \lim\limits_{j\to\infty} \frac{\sum\limits_{i=1}^{j-1} \sum\limits_{t=0}^{i-1} (b_{s_i+1-t}-1)\}}{n_{s_j}}+\lim\limits_{j\to\infty} \frac{\sum\limits_{i=0}^{m-1}(b_{s_j-i+1}-1)}{n_1-1+\sum\limits_{i=2}^{s_j+1} (b_i-1)} \\ &\leq& \lim\limits_{j\to\infty} \frac{1}{j-1}+0 =0.
\end{eqnarray*}
Consider $x\in T$ such that
$$
supp_{(a_n)}(x)\subseteq \{s_j:j\in\N\}.
$$
Now observe that for each $i\in \N\setminus A$ we have $i=n_{k}+r-1$ for some $r\in [1,b_{k+1}-1]$ and $k\notin \bigcup\limits_{j=m+1}^\infty [s_j-m+1,s_j]$, i.e., $k,k+1,k+2,\ldots,k+m-1\notin supp_{(a_n)}(x)$. Then for all $i\in \N\setminus A$, we have
\begin{eqnarray*}
\{d_ix\}=\{d_{n_{k}+r-1}x\} &=& \{ra_{k}x\} \\ &\leq&  r a_{k} \sum\limits_{i=k+m}^\infty \frac{c_i}{a_{i}} \\ &\leq&  \frac{r a_{k}}{a_{k+m-1}} \leq \frac{r}{b_{k+1}}\frac{a_{k+1}}{a_{k+m-1}}\leq\frac{1}{2^{m-2}}<\varepsilon.
\end{eqnarray*}
Therefore we can conclude that $x\in t^s_{(d_n)}(\T)$.

Let us fix a sequence $\xi = (z_i)\in \{0,1\}^\N$ and define $B^\xi = \bigcup\limits_{k=1}^{\infty}s_{2k+ z_k}$. In other words, this subset $B^\xi $ of $(s_j)$ is obtained by taking at each stage $k$ either $s_{2k}$ or $s_{2k+1}$ depending on the choice imposed by $\xi$. As obviously $B^\xi \ne B^\eta$
for distinct $\xi, \eta \in \{0,1\}^\N$, this provides an injective map given by
$$
\{0,1\}^\N \ni \xi \to B^\xi.
$$
Now for each $x^\xi\in\T$ with $supp_{(a_n)}(x^\xi)=B^\xi$, by the same argument we have $x^\xi\in t^s_{(d_n)}(\T)$. Since $|\{0,1\}^\N| = \mathfrak c$, we conclude that $|t^s_{(d_n)}(\T)|=\mathfrak{c}$.
\end{proof}

\begin{corollary}
Let $(\zeta_n)$ be the sequence defined in Eq (\ref{eqnonarith}). Then $|t^s_{(\zeta_n)}(\T)|=\mathfrak{c}$.
\end{corollary}
\begin{proof}
The proof follows directly by considering $a_n=n!$ in Theorem \ref{sconth04}.
\end{proof}

\begin{corollary}
For any arithmetic sequence $(a_n)$ satisfying the condition of Theorem \ref{sconth04}, $t_{(d_n)}(\T)\subsetneq t^s_{(d_n)}(\T)$ and $|t^s_{(d_n)}(\T)\setminus t_{(d_n)}(\T)|=\mathfrak{c}$.
\end{corollary}

\begin{proof}
Follows directly from \cite[Corollary 2.4, (iii)]{DG8}  and Theorem \ref{sconth04}.
\end{proof}

\noindent{\textbf{Acknowledgement:}} The first author as PI and the second author as RA are thankful to SERB(DST) for the CRG project (No. CRG/2022/000264) and the first author is also thankful to SERB for the MATRICS project (No. MTR/2022/000111) during the tenure of which this work has been done.\\

\end{document}